\documentclass[letterpaper, 10 pt, conference]{article}
\usepackage[margin=1.0in]{geometry}  

\usepackage{amsmath, amssymb,amsthm}
\usepackage{amsfonts}
\usepackage{graphicx, adjustbox}
\usepackage{enumerate}
\usepackage{caption}
\usepackage{ifthen}
\usepackage{multicol}
\usepackage{float}
\usepackage{fancyhdr}
\usepackage{cite}
\usepackage[usenames, dvipsnames]{color}
\usepackage{url}
 \usepackage{tikz}
 \usetikzlibrary{decorations.pathmorphing,shadings}
\usepackage{tkz-graph}
\usetikzlibrary{arrows}

\usepackage{subfigure}

\newtheorem{theorem}{Theorem}
\newtheorem{lemma}{Lemma}

\newtheorem{corollary}{Corollary}

\newboolean{showcomments}
\setboolean{showcomments}{true}

\newcommand{\comment}[1]{\ifthenelse{\boolean{showcomments}}
{\textcolor{Green}{(Comment: #1)}}{}}

\newcommand{\emma}[1]{  \ifthenelse{\boolean{showcomments}}
{\textcolor{Green}{Emma:  #1}}{}}
\newcommand{\hassan}[1]{\ifthenelse{\boolean{showcomments}}
{\textcolor{SkyBlue}{Hassan: #1}}{}}
\newcommand{\tim}[1]{\ifthenelse{\boolean{showcomments}}
{\textcolor{VioletRed}{Tim: #1}}{}}
\newcommand{\nithin}[1]{\ifthenelse{\boolean{showcomments}}
{\textcolor{Blue}{Nithin: #1}}{}}
\newcommand{\louis}[1]{\ifthenelse{\boolean{showcomments}}
{\textcolor{Orange}{Louis: #1}}{}}
\newcommand{\igor}[1]{\ifthenelse{\boolean{showcomments}}
{\textcolor{Red}{Igor: #1}}{}}

\newboolean{shownew}
\setboolean{shownew}{true}
\newcommand{\newtext}[1]{\ifthenelse{\boolean{shownew}}
{\textcolor{Red}{#1}}{}}

\newcommand{\bx}{\mathbf{x}}

\newcommand{\Koop}{\mathcal{U}}

\usepackage{array}
\newcolumntype{L}[1]{>{\raggedright\let\newline\\\arraybackslash\hspace{0pt}}m{#1}}
\newcolumntype{C}[1]{>{\centering\let\newline\\\arraybackslash\hspace{0pt}}m{#1}}
\newcolumntype{R}[1]{>{\raggedleft\let\newline\\\arraybackslash\hspace{0pt}}m{#1}}
\usepackage{subfig}

\begin{document}

\title{\LARGE \bf
	An operator-theoretic viewpoint to non-smooth dynamical systems: Koopman analysis of a hybrid pendulum}
\author{Nithin Govindarajan, Hassan Arbabi, Louis van Blargian, \\ Timothy Matchen, Emma Tegling and  Igor Mezi\'c %
	\thanks{ N. Govindarajan, H. Arbabi, L. van Blargian and T. Matchen are graduate students at Mechanical Engineering, University of California, Santa Barbara, Santa Barbara, CA 93106-5070, USA. }
	\thanks{ E. Tegling is with the School of Electrical Engineering, KTH Royal Institute of Technology, SE-100 44 Stockholm, Sweden. }
	\thanks{I. Mezi\'c is with the Faculty of Mechanical Engineering, University of California, Santa Barbara, Santa Barbara, CA 93106-5070, USA.}
}%
\maketitle

\begin{abstract}
	We apply an operator-theoretic viewpoint to a class of non-smooth dynamical systems that are exposed to event-triggered state resets. The considered benchmark problem is that of a pendulum which  receives a downward kick under certain fixed angles. The pendulum is modeled as a hybrid automaton and is analyzed from both a geometric perspective and the formalism carried out by Koopman operator theory. A connection is drawn between these two interpretations of a dynamical system by means of establishing a link between the spectral properties of the Koopman operator and the geometric properties in the state-space.
\end{abstract}
\section{Introduction}
\label{sec:intro}

A considerable number of dynamical systems in engineering practice are essentially hybrid in nature. These systems typically model non-smooth phenomena such as impact, collision, and switching between several discrete modes. Under the hybrid automaton framework, these discontinuities are often expressed in terms of guard conditions, state resets, and switching between several ``system modes''. The imposition of such conditions result in so-called piecewise-smooth dynamical systems of which the orbits are characterized by smooth evolutions, interrupted by discrete jumps. 

With the possibility of discontinuous orbits, it may not always be convenient to characterize the state-space geometry of a hybrid system in terms of its trajectories. A more general viewpoint to take is to consider the evolution of \emph{ensembles of initial conditions}, or sets being propagated under the flow. In this paper, we embrace this philosophy by viewing a special class of hybrid systems from an operator-theoretic point of view. In this approach, instead of focusing on trajectories, the evolution of functions defined on the state-space is considered. Our analysis is based-on the machinery of Koopman operator theory and looks at the so-called “dynamics of observables”. A remarkable feature to this approach is that the dynamics can be interpreted as linear in the space of observables, irrespective of the underlying properties of the dynamical system in the state-space. The Koopman (semi-)group is a one-parameter family of infinite-dimensional linear operators, and this  allows one to exploit the tools of spectral operator theory. Analysis of nonlinear flow fields through the spectral properties of the Koopman operator has already been carried out under various settings\cite{Budisic2012}. From an applied context, these approaches have been particularly useful in describing dynamically relevant modes in fluid flows \cite{Mezic2012,Rowley2009}, coherency in power systems \cite{Susuki2011}, and energy efficiency in buildings \cite{Georgescu2015}.

The purpose of this paper is to illustrate how Koopman analysis can also be applied to certain classes of hybrid systems which consist of a single discrete mode, but are subjected to guard conditions and state resets. The emphasis will be on a specific benchmark pendulum system that is subjected to downward ``kicks'' under fixed angles. For this particular example, we show how certain important geometric structures, pertinent to the underlying flow field, are recovered from the spectral properties of the operator.  Overall, two cases are considered: (i) the undamped case where the continuous part of the system is Hamiltonian, and (ii) a damped case where the pendulum is exposed to viscous damping. 

The paper is organized as follows. Section~\ref{sec:kooptheory} reviews the essentials of Koopman operator theory and introduces the specifics on the hybrid pendulum. Section~\ref{sec:mapping} analyzes the undamped case from a geometric perspective. Section~\ref{sec:numerical} then views the undamped system from the Koopman operator perspective. In sections~\ref{sec:damping} and \ref{sec:dampedkoopman}, the same are done respectively for the damped case. A reflection on the obtained results is covered in the conclusions. Proofs of certain theorems are included in the appendix.

\section{Preliminaries: Koopman operator theory and the ``hybrid pendulum''} \label{sec:kooptheory}

In section~\ref{sec:Koopman_intro} we review certain basics of Koopman operator theory \cite{Budisic2012, Mezic2005}. The details on the ``hybrid pendulum'' are covered in  section~\ref{sec:setup}.

\subsection{Koopman operator theory}
\label{sec:Koopman_intro}

Let $X\subset \mathbb{R}^N$ denote the state-space and $\boldsymbol{S}^t: X\mapsto X$ a flow map satisfying the (semi-)group properties: $\boldsymbol{S}^t\circ \boldsymbol{S}^s(\boldsymbol{x}) = \boldsymbol{S}^{t+s}(\boldsymbol{x})$, $\boldsymbol{S}^0(\boldsymbol{x})= \boldsymbol{x}$.  Now consider an observable $g: X\mapsto \mathbb{C}$, for some fixed $t\in \mathbb{R}$, the Koopman operator is defined by, 
\begin{equation}
\left[  \Koop^t g \right](\boldsymbol{x}) := g \circ \boldsymbol{S}^t (\boldsymbol{x}) \label{eq:koopdefcont}
\end{equation}
\eqref{eq:koopdefcont} permits an alternative representation of a dynamical system in which one looks at the dynamics of observables, i.e. the original system expressed by the tuple $(X, \boldsymbol{S}^t, t)$ can be alternatively represented by $(\mathcal{G}, \Koop^t, t)$, where  $\mathcal{G}$ denotes the space of observables. A remarkable aspect to this representation $(\mathcal{G}, \Koop^t, t)$ is that $\Koop^t$ is always a \emph{linear operator}, irrespective of the original properties in the state-space. This linear description of a nonlinear, and possibly, non-smooth system is obtained through ``lifting''  the dynamics on the state-space to a higher, infinite-dimensional space of functions. 

\subsubsection*{Koopman eigenfunctions/-distributions}
The linearity of $\Koop^t$ allows one to exploit the machinery from spectral operator theory. In particular we may define eigenfunctions for these operators. A nonzero function $\phi_{\lambda}\in\mathcal{G}$ is called a Koopman eigenfunction if it satisfies,
\begin{equation}
\left[ \Koop^t \phi_{\lambda}\right] (\boldsymbol{x}) = e^{\lambda t} \phi_{\lambda} (\boldsymbol{x})   \label{eq:conteigfunc}
\end{equation}
for some eigenvalue $\lambda \in \mathbb{C}$. We remark that, depending on what norm and measure is used on $X$, the expression $\phi_{\lambda} \ne 0$ must be interpreted in an almost everywhere sense, e.g.  for square-integrable functions with respect to measure $\mu$, we have the condition: $\left\| \phi_{\lambda} \right\|_{2,\mu}\ne 0$, where $\left\| v \right\|_{2,\mu} := ( \int_X |v(\boldsymbol{x})|^2 d\mu  )^{\frac{1}{2}}$. 

Given the infinite-dimensional nature of the operator, $\Koop^t$ may also contain continuous spectrum. In that case, one can extend the notion of eigenfunctions in an appropriate weak sense using the concept of distributions. These generalized objects, referred to as eigendistributions, satisfy the relation:
\begin{equation}
\int_X \left[\Koop^t \phi_{\lambda}\right] (\boldsymbol{x}) w(\boldsymbol{x}) d\mu = e^{\lambda t}  \int_X \phi_{\lambda}(\boldsymbol{x})  w(\boldsymbol{x})  d\mu \label{eq:eigenmeasure}
\end{equation}
where $w(\boldsymbol{x})$ is some arbitrary test function on $X$. Indeed, it follows that all eigenfunctions are eigendistributions, but the converse does not hold true. 

Eigenfunctions/-distributions are preserved under conjugacy. If $\mathbf{S}^t:X\mapsto X$, $\mathbf{R}^t: Y\mapsto Y$ are two topologically conjugate dynamical systems under the homeomorphism $\boldsymbol{h}: X \mapsto Y$, i.e. $ \boldsymbol{h} \circ \mathbf{S}^t(\boldsymbol{x}) = \mathbf{R}^t \circ \boldsymbol{h}(\boldsymbol{x})$, and if $\phi_{\lambda}$ is an eigenfunction/eigendistribution of $\Koop^t_{R}$, then so should  $\phi_{\lambda}\circ \boldsymbol{h} $ be an eigenfunction/-distribution of $\Koop^t_{S}$ \cite{Budisic2012}.  

Koopman eigenfunctions are directly related to the geometric state-space description of a dynamical system in the following sense. Let: 
$$\Psi^c_{\phi_{\lambda}} := \left\lbrace{\boldsymbol{x}\in X: \phi_{\lambda}(\boldsymbol{x}) = c} \right\rbrace$$
denote a specific level-set of an eigenfunction. Then the mapping of the set $\Psi^c_{\phi_{\lambda}}$ forward under the flow yields the relation:
\begin{equation} 
\boldsymbol{S}^t\left( \Psi^c_{\phi_{\lambda}}\right) = \Psi^{c \exp(\lambda t)}_{\phi_{\lambda}} \label{eq:levelsetprop}
\end{equation}
The interpretation of \eqref{eq:levelsetprop} is that the level-sets of $\phi_{\lambda}$ characterize how specific ensembles of initial conditions are propogated under the flow. It is exactly this specific property which allow us to analyze the geometric properties of the state-space from an operator theoretic context.

\subsubsection*{Projection operators}

Given an observable $g$, one may obtain the projection of this observable onto the fixed space (i.e. eigenspace at eigenvalue $\lambda=0$)through evaluating the infinite time averages of observable-traces:
\begin{equation}
g^*(\boldsymbol{x}) = \displaystyle \lim_{t\rightarrow \infty} \frac{1}{t} \int^t_0 \left[ \Koop^t g  \right] (\boldsymbol{x}) dt
\label{eq:avg}
\end{equation} 
By Birkhoff's ergodic theorem \cite{Petersen1989}, the integral \eqref{eq:avg} is known to converge a.e. for integrable functions with respect to the invariant measure of the system. Through adding a weighting term, one may obtain also  projections of $g$ onto eigenspaces other than zero. In general, we define a projection operator:
\begin{equation}
\left[ \mathcal{P}^{\lambda}g \right] (\boldsymbol{x}) = \displaystyle \lim_{t\rightarrow \infty} \frac{1}{t} \int^t_0 e^{-\lambda \tau} \left[ \Koop^{\tau} g  \right] (\boldsymbol{x}) d\tau \label{eq:laplace}
\end{equation} 
where the right-hand side of \eqref{eq:laplace} is generally known as the \emph{Laplace average} of $g$ \cite{Budisic2012}. One can verify through substitution that $\mathcal{P}^{\lambda}g$ is indeed an eigenfunction of Koopman at eigenvalue $\lambda$, provided the improper integral converges.

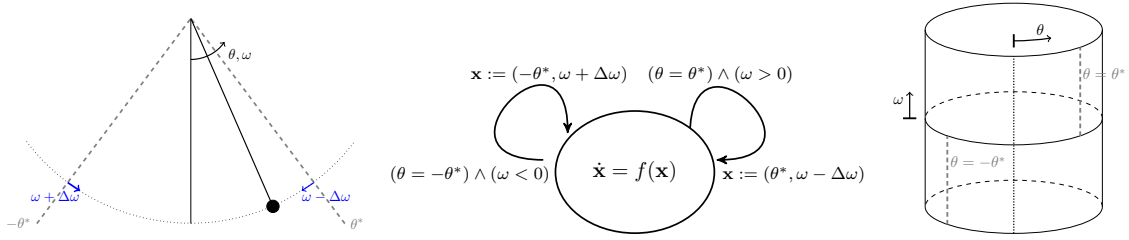
\begin{figure*}[t!]
	\centering
	\begin{adjustbox}{max width=0.3\textwidth}
		\begin{tikzpicture}[line/.style={thick},
thetaline/.style={very thick,style = dashed,color=gray},
  ball/.style={draw,circle,inner sep=0,minimum size=0.3cm,fill} ]

	\draw[thetaline] (1.25,0) -- (5,5) node[pos=0,left] {$-\theta^*$}; 
	\draw[thetaline] (8.75,0) -- (5,5) node[pos=0,right] {$\theta^*$};
		\draw[style = dotted] (1,2) arc (217:323:5) ;
	\draw (5,5) -- (5,0) ;
	\draw[line] (7,0.42) -- (5,5) ;
	\node[ball] at (7,0.42) {};
	\draw[->,style=thick] (5,4) arc (270:325:1) ;
	\node at (6.2,4.1) {$\theta, \omega$};
	\draw[->, style=thick, color = blue]  (8,1) -- (7.7,0.8);  
	\node[color = blue] at (8.3,0.65){$\omega-\Delta \omega$};
	\draw[->, style=very thick, color = blue]  (2,1) -- (2.3,0.8);  
	\node[color = blue] at (1.7,0.65){$\omega+\Delta \omega$};

\end{tikzpicture}
	\end{adjustbox}  	
	\begin{adjustbox}{max width=0.4\textwidth}
		\begin{tikzpicture}
\SetGraphUnit{2} 
\tikzset{VertexStyle/.style = {minimum size=0.8cm,
                               font=\Large\bfseries},thick, color = white} 
\Vertex{1} \EA(1){2} 
\Loop[dist=2cm,dir=NOWE](1)  
\Loop[dist=2cm,dir=NOEA](2)  

\draw[style = thick, fill=white] (1.1,-0.2) ellipse (1.3cm and 1.0cm) node[color = black] {$\dot{\bx} = f(\bx)$} ;
\node[color = black,align = center,font = \footnotesize] at (-1.6,-0.25) {$(\theta=-\theta^*) \wedge (\omega<0)$};
\node[color = black,align = center,font = \footnotesize] at (-0.3,1.4) {$\bx :=(-\theta^*,\omega + \Delta\omega)$};
\node[color = black,align = center,font = \footnotesize] at (2.5,1.4) {$(\theta=\theta^*) \wedge (\omega>0)$};
\node[color = black,align = center,font = \footnotesize] at (3.7,-0.25) {$\bx :=(\theta^*,\omega - \Delta\omega)$};
\end{tikzpicture}

	\end{adjustbox}  
	\begin{adjustbox}{max width=0.2\textwidth}
		\begin{tikzpicture}[line/.style={thick} ]

\draw (0,0) arc (180:360:2 and 0.6);
\draw[dashed] (4,0) arc (0:180:2 and 0.6);

\draw (0,2) arc (180:360:2 and 0.6);
\draw[dashed] (4,2) arc (0:180:2 and 0.6);

\draw (4,4) arc (0:360:2 and 0.6);

\draw (0,0) -- (0,4) ;
\draw (4,0) -- (4,4) ;

\draw[style = densely dotted, style=thick] (2,-0.6) -- (2,3.4) ;

\draw[style = very thick] (2,3.6) -- (2,3.9) ;
\draw[->,style=thick] (2,3.75) arc (270:295:2 and 0.7);
	\node at (2.6,4.0) {$\theta$};
	
\draw[style = very thick] (-0.5,2) -- (-0.2,2) ;
	\draw[->,style=thick] (-0.35,2) -- (-0.35, 2.6);
	\node at (-0.6,2.4) {$\omega$};


\draw[style = very thick, style = densely dashed, color = gray] (3.5,1.6) -- (3.5,3.6) ;
\draw[style = very thick, style = densely dashed, color = gray] (0.5,-0.4) -- (0.5,1.6) ;
	\node[color = gray,align = left]  at (4.05,3) {$\theta = \theta^*$};
	\node[color = gray,align = left]  at (1.2,1) {$\theta = -\theta^*$};

\end{tikzpicture}
	\end{adjustbox}
	\caption{The  hybrid pendulum: (left) upon passing the angle $\pm \theta^*$ from below the pendulum experiences a change in angular velocity by $\Delta \omega$, (center) the corresponding hybrid automaton representation, (right) the state-space of the system. }  \label{fig:pendulum}
\end{figure*}

\subsection{The hybrid pendulum}
\label{sec:setup}

Consider a mathematical pendulum with length $l$ and mass $m$. In the absence of damping or external forcing, the equations of motion for this system are formed by defining $\boldsymbol{x} := (\theta, \omega) \in \mathbb{S}^{1} \times \mathbb{R} =: X$ and $\boldsymbol{f}: X \mapsto \mathbb{R}^2$ such that: 
\begin{displaymath}
\dot{\boldsymbol{x}} = \boldsymbol{f}( \boldsymbol{x}) = \begin{bmatrix} \omega \\ -(g/l) \sin \theta \end{bmatrix}.
\label{eq:motion}
\end{displaymath}
Now suppose that the pendulum experiences an instantaneous backwards ``kick'' when passing through the given angles $\pm\theta^*$.  This kick is modeled by an instantaneous change in angular velocity $\Delta \omega >0$. In the hybrid automaton notation, the kick is included in the model as a reset map $\mathcal{R}: X \mapsto X$ defined by,
\begin{displaymath}
\mathcal{R}(\theta, \omega) = \begin{cases} (-\theta^*, \omega+\Delta \omega) &\mbox{if } (\theta = -\theta^*) \wedge (\omega <0) \\ 
(\theta^*, \omega-\Delta \omega) &\mbox{if } (\theta = \theta^*) \wedge (\omega >0)  \end{cases}, 
\label{eq:reset}
\end{displaymath}    
where, for ease of notation, we have incorporated the guard conditions as well. Note that the reset only occurs upon passing through $\pm \theta^*$ from below, so that the kick is always directed towards the stable equilibrium point of the pendulum. We also remark that according to this formulation, no reset occurs when the pendulum only grazes the ``kicking surfaces'' at $\pm \theta^*$. The situation is illustrated in Fig.~\ref{fig:pendulum}, where we also show the corresponding hybrid automaton representation and the state-space.

\subsubsection*{Normalized equations}	
To simplify our analysis, it is convenient to normalize the state $\boldsymbol{x}$ by dividing the angular velocity $\omega$ by the kick strength $\Delta \omega$, i.e.
\begin{equation}
\begin{bmatrix}
\dot{\theta} \\ \dot{p} 
\end{bmatrix} = \begin{bmatrix} \mu_1 p  \\ -\left(\mu_2^2 / \mu_1 \right) \sin \theta \end{bmatrix}. 
\label{eq:pendnorm}
\end{equation}
and
\begin{equation}
\mathcal{R}(\theta, p) = \begin{cases} (-\theta^*, p + 1) &\mbox{if } (\theta = -\theta^*) \wedge (p <0) \\ 
(\theta^*, p -1) &\mbox{if } (\theta = \theta^*) \vee (p >0)  \end{cases},  
\label{eq:quardnorm}
\end{equation}  
where $p=\omega/\Delta\omega$ denotes the normalized momentum. The equations \eqref{eq:pendnorm},  \eqref{eq:quardnorm}, are parametrized in terms of $\mu_1 := \Delta \omega>0$  (the kick strength) and  $\mu_2 := \omega_n := \sqrt{g/l}>0$ (the natural frequency of the linearized pendulum). Note that the continuous part of the (undamped) hybrid pendulum is Hamiltonian. One can verify that the function:
\begin{equation}
\mathcal{H}(\theta,p) := \frac{1}{2}  \left(\frac{\mu_1}{ \mu_2} p\right)^2 + 1-\cos \theta \label{eq:hamdef}
\end{equation}
constitute an invariant for the flow of \eqref{eq:pendnorm}. This particularly implies that, in between state resets, the trajectories of the pendulum are confined to level sets of \eqref{eq:hamdef}. 
\subsubsection*{The damped case}	
We will also look into the effects of weak viscous damping on the system. 
The continuous part of the hybrid system in that case  is replaced by
\begin{equation}
\begin{bmatrix}
\dot{\theta} \\ \dot{p} 
\end{bmatrix} = \begin{bmatrix} \mu_1 p  \\ -\left(\mu_2^2 / \mu_1 \right) \sin \theta - k p \end{bmatrix} \label{eq:penddamped}
\end{equation}
where $k>0$ is the viscous damping coefficient.

\section{The undamped hybrid pendulum: classical geometric analysis}
\label{sec:mapping}

In this section, the state-space of the undamped hybrid pendulum is described from a geometric perspective.

\subsection{The Poincar\'{e} map}
The trajectories of the freely oscillating pendulum are confined to the level sets of \eqref{eq:hamdef}. Through this observation, it is clear that for initial conditions belonging to the set
\begin{equation}
\mathcal{A}_1 :=  \left\{(\theta, p)\in X:  \mathcal{H}(\theta,p) < \mathcal{H}(\theta^*,0)  \right\} \label{eq:simple}
\end{equation}
the behavior of the hybrid pendulum is exactly identical to that of the freely oscillating pendulum. 

The more distinctive behavior can be found only in the region: $\mathcal{H}(\theta,p) \geq \mathcal{H}(\theta^*,0)$. In this part of the state-space, the pendulum gets ``kicked'' at least once in its orbit for almost any initial condition. The only initial conditions that never get kicked here are those that lie exactly on the homoclinic orbit of the unstable fixed point, and additionally satisfy $\theta>\theta^*$, $p>0$ or $\theta<-\theta^*$, $p<0$.  

The dynamics of the kicked region can be fully understood from a discrete map defined on the kicking surfaces. Given that the orbits between two consecutive impacts are uniquely determined by the momentum at $\pm \theta^*$, a map can be defined to describe the (normalized) momentum $p\big|_{\theta=\pm \theta^*}$ of the pendulum right \emph{before} the next impact. To do this thoroughly, we first describe those points on the kicking surfaces that directly get mapped into the homoclinic orbit. Let:
\begin{equation}
p_{cr}:= \frac{\mu_2}{\mu_1} \sqrt{ 2 + 2 \cos\theta^*}
\end{equation}
denote the critical momentum required to be on the homoclinic orbit and assign the variable $\gamma$ to be the state of the pendulum at the unstable fixed point. The Poincar\'{e} map $T: \lbrace{\mathbb{R}\cup \gamma}\rbrace \mapsto \lbrace{\mathbb{R}\cup \gamma}\rbrace$ is defined by
\begin{equation}
T(p) = \begin{cases} p+1 & p< -(p_{cr}+1) \\ 
|p + 1| & -(p_{cr}+1) < p \leq 0 \\
0 & p=0 \\
-|p - 1| & 0 <p < p_{cr}+1 \\
p - 1& p_{cr}+1 < p \\
\gamma & p=\gamma \mbox{ or }p=\pm (p_{cr}+1) 
\end{cases}
\label{eq:redmap}
\end{equation}

\subsection{Asymptotic dynamics of the hybrid pendulum} \label{sec:asymptotics}
To fully describe the asymptotic dynamics of the pendulum in the region: $\mathcal{H}(\theta,p) \geq \mathcal{H}(\theta^*,0)$, we first state the following result about the map \eqref{eq:redmap}.
\begin{lemma}
	\label{thm:periodicorbit}
	The map $T: \lbrace{\mathbb{R}\cup \gamma}\rbrace  \mapsto \lbrace{\mathbb{R}\cup \gamma}\rbrace$ defined  by (\ref{eq:redmap}) has the following asymptotic properties:
	\begin{enumerate}[(i)]
		\item If $p \in D:= \left\{ p\in \mathbb{R}:  p = \gamma \wedge p =  \pm(p_{cr}+k), k\in\mathbb{N}  \right\}$, then there exists $M>0$, such that:
		\begin{displaymath}
		T^n(p) = \gamma, \quad \forall n > M
		\end{displaymath}
		\item If $p\notin D$, then there exists a $M>0$, such that:
		\begin{displaymath}
		T^n(p) \in \left[-1, 1\right], \quad \forall n > M
		\end{displaymath}
	\end{enumerate}
\end{lemma}
\begin{proof}
	(i) can be established by computing the pre-images $T^{-{k}}(\lbrace \gamma \rbrace)$ for $k\in\mathbb{N}$.
	To show that (ii) is true, observe at first that $T(\left[-1, 1\right]) = ( -1, 1)$,  which shows that $\left[-1, 1\right]$ is a positively invariant set. Now consider any $p\notin D$ with $|p|>1$, then $|T(p)| = |p| - 1$. Using induction, we may show that: 
	$$
	| T^n(p) | = |p| - n, \quad |p| > n \label{eq:proofrel1}
	$$
	from where it follows that $T^n(p)$ enters the interval $\left[-1, 1\right]$ in a finite number of iterations.
\end{proof}

Lemma~\ref{thm:periodicorbit} states that the interval $\left[-1, 1\right]$ is an attracting set for (\ref{eq:redmap}). Furthermore, it states that the trajectories enter the interval $\left[-1, 1\right]$ in a finite number of iterations. The interior of the attracting set is composed of:
\begin{enumerate}[(i)]
	\item a fixed point at $p=0$.
	\item an uncountable family of period-2 cycles of the form:
	\begin{equation}
	\{p_1,p_1-1\},  \quad p_1 \in (0,1)  \label{eq:period2orbits}
	\end{equation}
\end{enumerate}

These results on the map \eqref{eq:redmap} are related to the actual hybrid system in the following way. The interval $[-1,1]$ from  Lemma~\ref{thm:periodicorbit} corresponds to the set
\begin{equation}
\label{eq:attractingset}
\mathcal{A}_2 :=  \left\{(\theta, p)\in X: H_-  \leq \mathcal{H}(\theta,p) \leq H_+, |\theta| \leq \theta^* \right\} 
\end{equation}
where:
\begin{equation}
H_- := \mathcal{H}\left(\theta^*,0\right), \quad H_+ := \mathcal{H}\left(\theta^*,1\right). \label{eq:defHplusmin}
\end{equation}
This set is foliated by an uncountable family of limit cycles. In correspondence with the period-2 cycles of the map (\ref{eq:redmap}), the limit cycles can as well be parametrized by  $\{p_1,p_1-1\}$, $p_1 \in (0,1)$. We have the following relationship between the original coordinates and the limit cycle in which the system is on:
\begin{equation}
p_1(\theta,p) = \begin{cases} \frac{\mu_2}{\mu_1 } \sqrt{2(\mathcal{H}(\theta,p) - 1 + \cos \theta^*)} & p\geq 0 \\
1- \frac{\mu_2}{\mu_1 } \sqrt{2(\mathcal{H}(\theta,p) - 1 + \cos \theta^*)}& p<0 \end{cases} \label{eq:action}
\end{equation}
where $(\theta,p)\in \mathcal{A}_2$. In summary, we have the following result.
\begin{theorem} \label{thm:main}
	The trajectories of the hybrid pendulum, starting from almost everywhere in the region': $\mathcal{H}(\theta,p) \geq \mathcal{H}(\theta^*,0)$, enter into a discontinuous periodic orbit in finite time. 
\end{theorem}
\begin{proof} Trajectories in the interior of the set (15) are already in a discontinuous periodic orbit. From lemma 1 it follows that almost all other trajectories eventually enter one of these periodic orbits. A measure zero set of trajectories get kicked into the homoclinic orbit, hence the statement almost everywhere.
\end{proof}

\subsection{Basin of attraction}

The basin of a specific limit cycle $p_1$ can be found by repeatedly computing the pre-images of the map (\ref{eq:redmap}). From this construction, we observe that the basin of every limit cycle is a measure zero set. 
Fig.~\ref{fig:basin} shows the basin of the limit cycle at $p_1 = 0.7$ for three different values of $p_{cr}$. 

\begin{figure*}[t!]
	\centering
	\subfigure[$p_{cr} =  1.65$]{ \includegraphics[width=.32\textwidth]{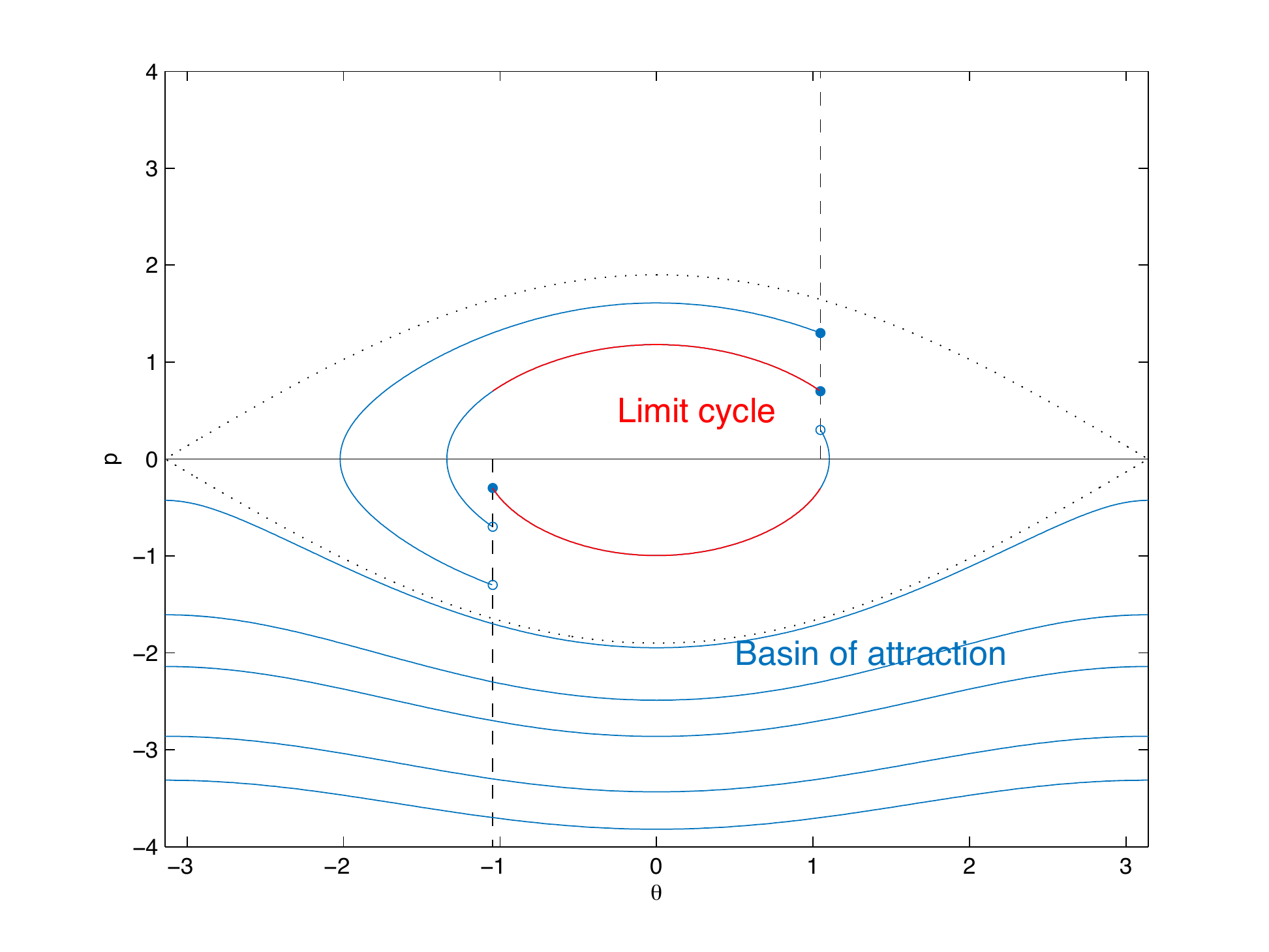} }
	\subfigure[$p_{cr} =  1.7$]{\includegraphics[width=.32\textwidth]{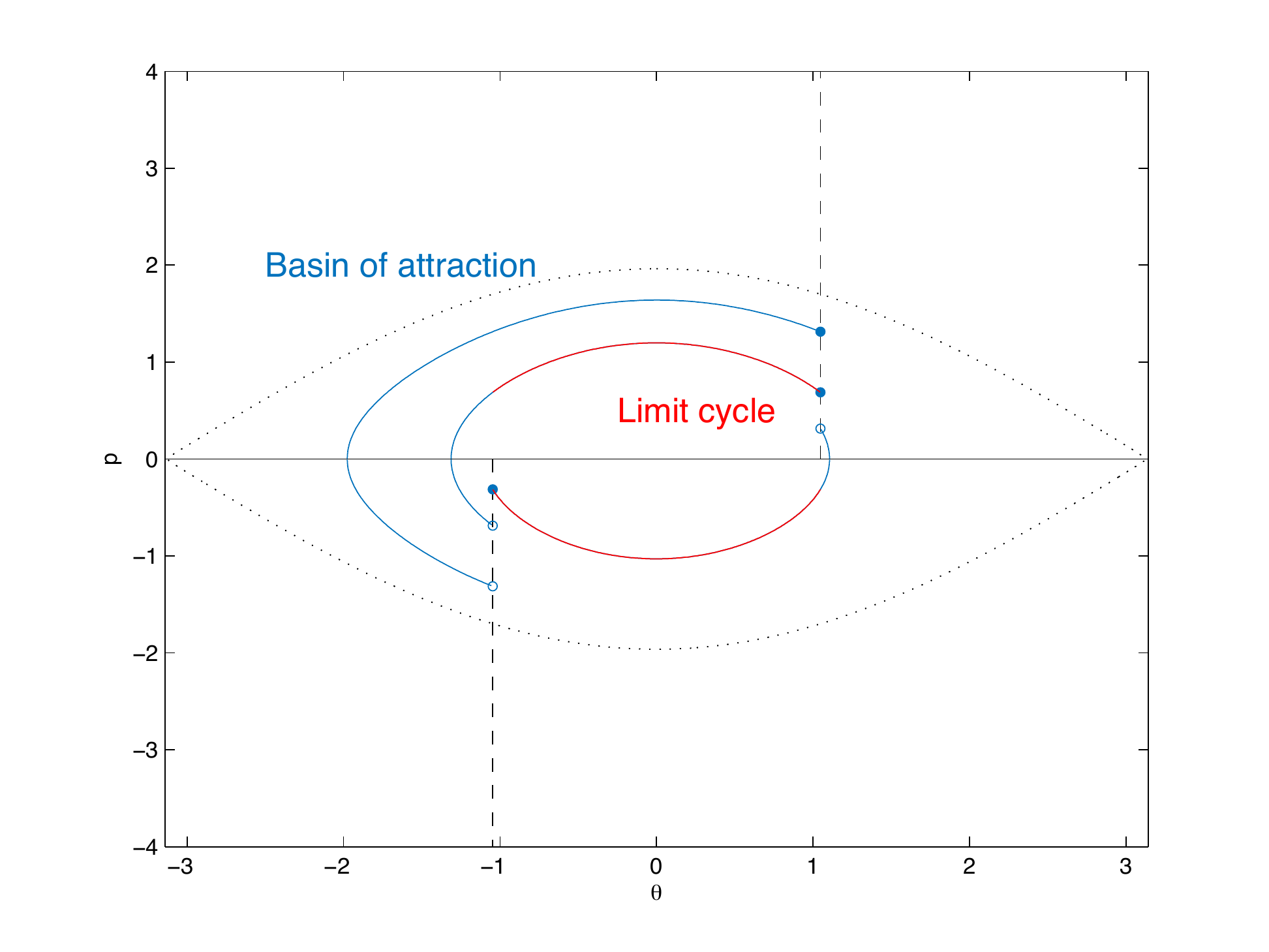} 	}
	\subfigure[$p_{cr} =  1.73$]{\includegraphics[width=.32\textwidth]{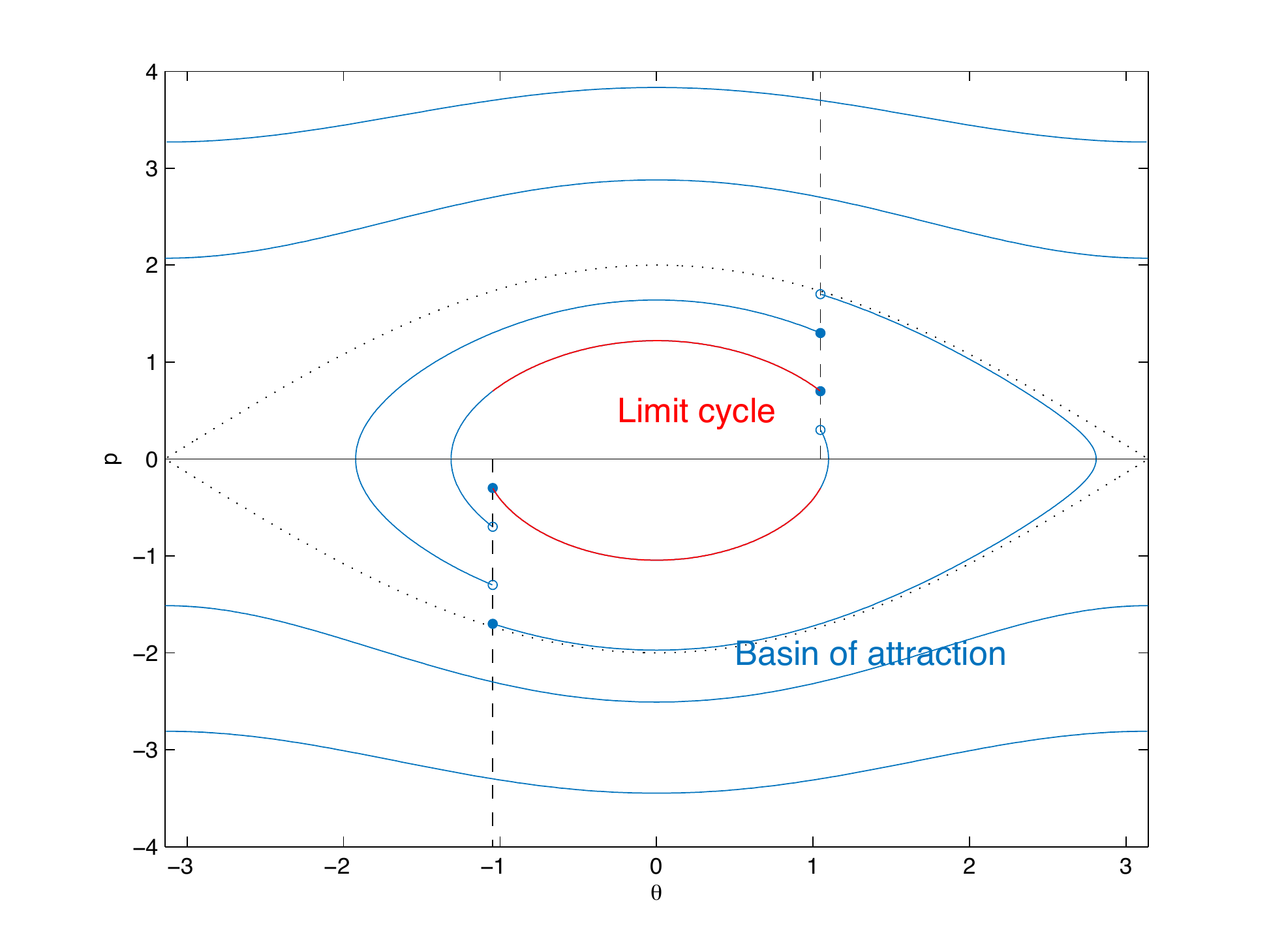}  }
	\caption{Basin of attraction for the limit cycle with $p_1 = 0.7$.} \label{fig:basin}
\end{figure*}

\subsection{Action-angle coordinates}

If the pendulum is released at $\theta=-\theta^*$ with an initial momentum $p_1\in (0,1)$, then the time required to reach a certain  $\theta \in [-\theta^*, \theta^*]$ is determined by the elliptic integral:
$$ \Gamma[\theta,p_1] = \int_{-\theta^*}^{\theta}\frac{1}{\mu_2}\left[ 2 \left(\frac{1}{2}\left( \frac{\mu_1}{\mu_2} p_1 \right)^2  - \cos \theta^* + \cos \xi \right) \right]^{-\frac{1}{2}}\mathrm{d}\xi$$
The function $\Gamma[\theta,p_1]$ permits us to define  action-angle coordinates for the set \eqref{eq:attractingset}. 

The period of a specific limit cycle $\{p_1, p_1 -1 \}$ is given by the formula:
\begin{equation}
P[p_1] = \Gamma[\theta^*,p_1] + \Gamma[\theta^*,1-p_1]
\end{equation}
On every limit cycle $\{p_1, p_1 -1 \}$, we can assign a phase coordinate $\psi\in[0,2\pi)$ such that: $\psi = 0$ at $(\theta, p) = (-\theta^*,p_1)$. This is done as follows: let $o(p_1)$ denote the orbit of a specific limit cycle, i.e.
$$ o(p_1) := \left\{  (\theta, p) \in X: S^t(-\theta^*,p_1) = (\theta, p) \mbox{ for some }t\geq0    \right\}  $$
Then, the phase on $o(p_1)$ can be defined as:
\begin{equation}
\psi=  \frac{1}{P[p_1]}\begin{cases}  \Gamma[\theta,p_1] & p > 0   \\ \Gamma[\theta^*,p_1] + \Gamma[\theta^*-\theta,1-p_1] & p<0  \end{cases} \label{eq:normphase}
\end{equation}
where $(\theta,p) \in o(p_1)$. 

The formulas \eqref{eq:normphase} together with \eqref{eq:action} define action-angle coordinates for the interior of the set \eqref{eq:attractingset}. That is, under the coordinate transformation $(I,\psi)= \boldsymbol{h}(\theta,p)$, where:
\begin{subequations}
	\begin{equation}
	h_1(\theta,p) = p_1(\theta,p)  \label{eq:conj1}
	\end{equation}
	\begin{multline}
	h_2(\theta,p)  = \\ \frac{1}{P[p_1(\theta,p) ]}\begin{cases}  \Gamma[\theta,p_1(\theta,p)] & p > 0   \\ \Gamma[\theta^*,p_1(\theta,p) ] + \Gamma[\theta^*-\theta,1-p_1(\theta,p) ] & p<0  \end{cases}
	\end{multline} \label{eq:bijection}
\end{subequations}
the set $\mathcal{A}_2$ is mapped onto the set 
\begin{equation}
Y:=(0,1)\times \mathbb{S}^1 \label{def:domainY}
\end{equation}
under which the flow is simply 
\begin{equation} \boldsymbol{R}^t(I, \psi) = \left(I, \left( \Omega[I] t + \psi\right) \mod 2\pi \right)  \label{eq:flowactionangle}\end{equation}
where $\Omega[I] := {2\pi }/{P[I]}$.

The kicked pendulum has an invariant measure whose support is restricted to the set $\mathcal{A}_2$. In fact, under the bijection \eqref{eq:bijection} the dynamics on the set $\mathcal{A}_2$ is conjugate to the Lebesgue measure-preserving system in \eqref{eq:flowactionangle}. Hence, if $\mu_{Y}$ denotes the Lebesgue measure for the domain \eqref{def:domainY}, then 
\begin{equation}
\mu_{\mathcal{A}_2} = \mu_{Y} \circ \boldsymbol{h} \label{eq:invmeas}
\end{equation}    
is an invariant measure for the hybrid system on \eqref{eq:attractingset}.

\section{The undamped hybrid pendulum: Koopman analysis}
\label{sec:numerical}

In this section, the undamped hybrid pendulum is analyzed from the Koopman operator theory perspective.

\subsection{Eigenspace of Koopman at $\lambda=0$} \label{sec:eigspacezero}

\begin{figure*}[t!]
	\centering
	\includegraphics[width=.46 \textwidth]{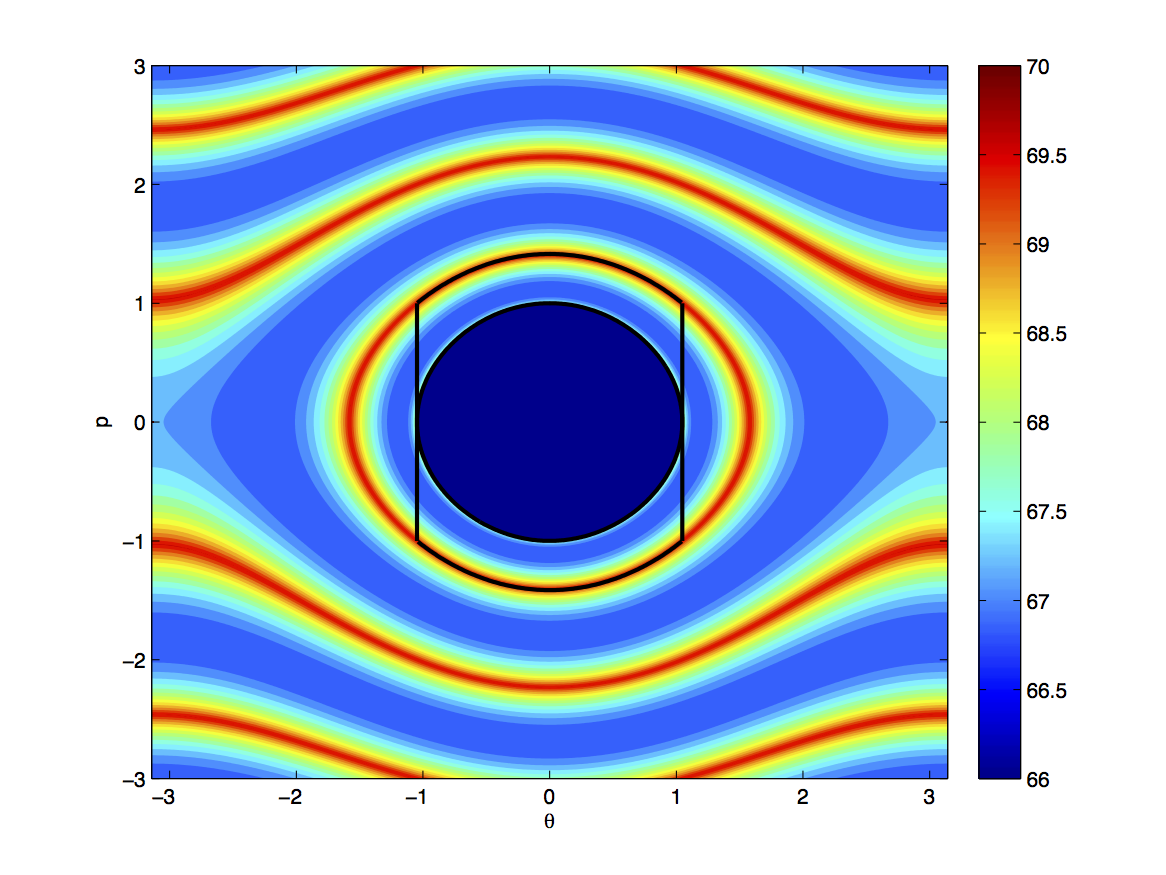} 
	\includegraphics[width=.46 \textwidth]{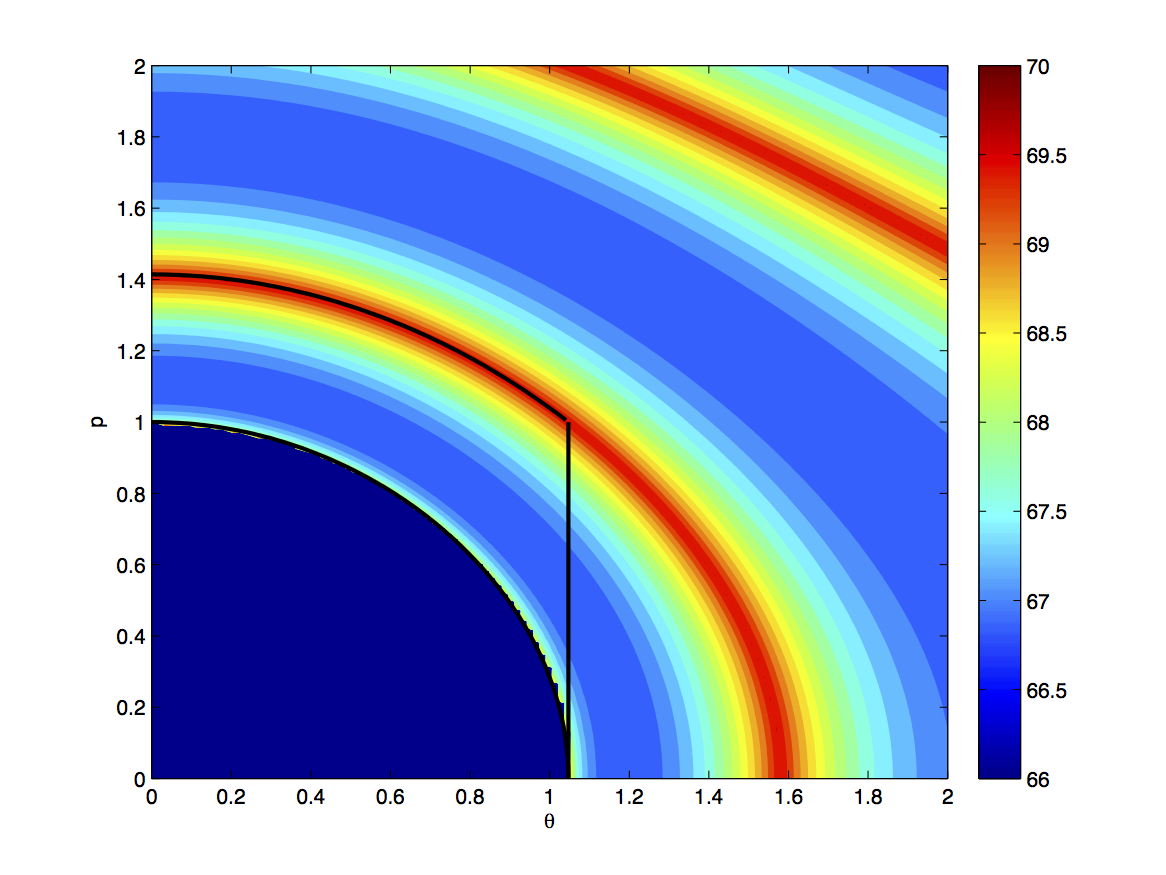}
	\includegraphics[width=.46 \textwidth]{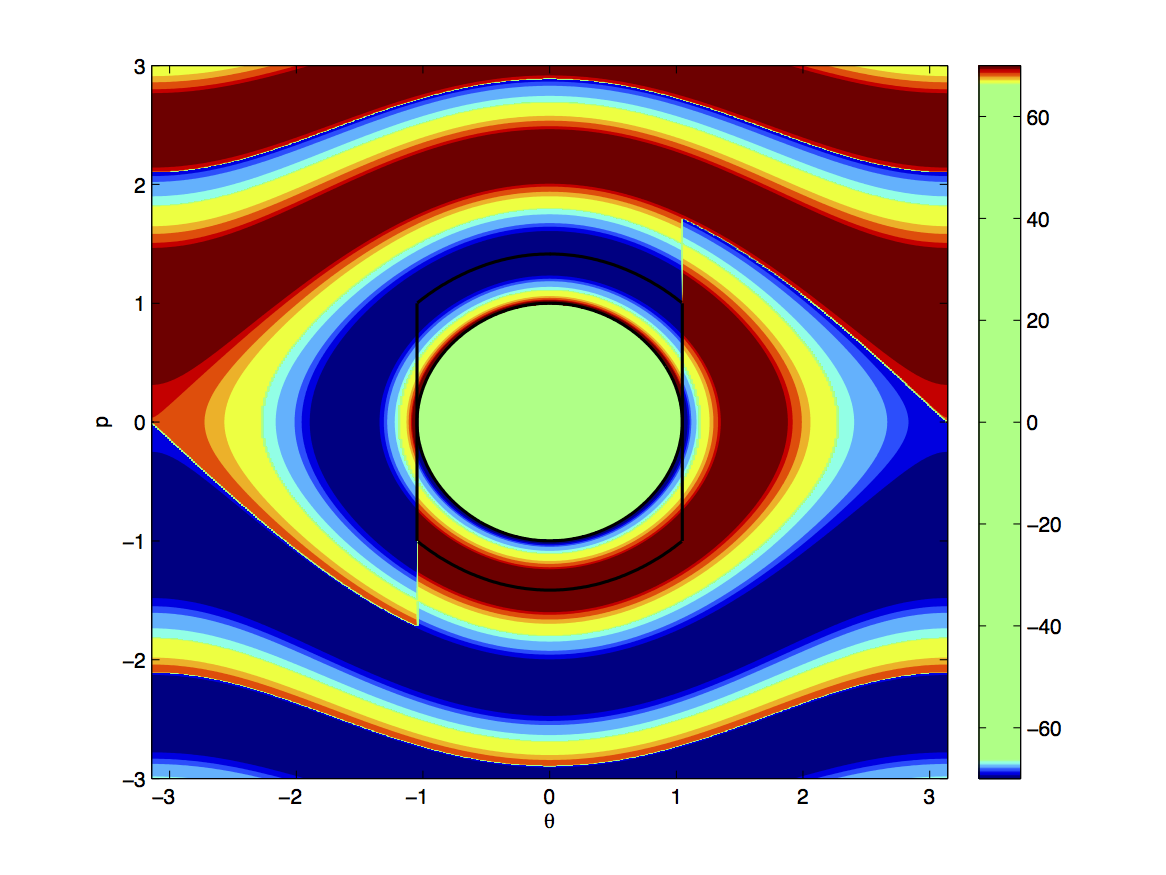}
	\includegraphics[width=.46 \textwidth]{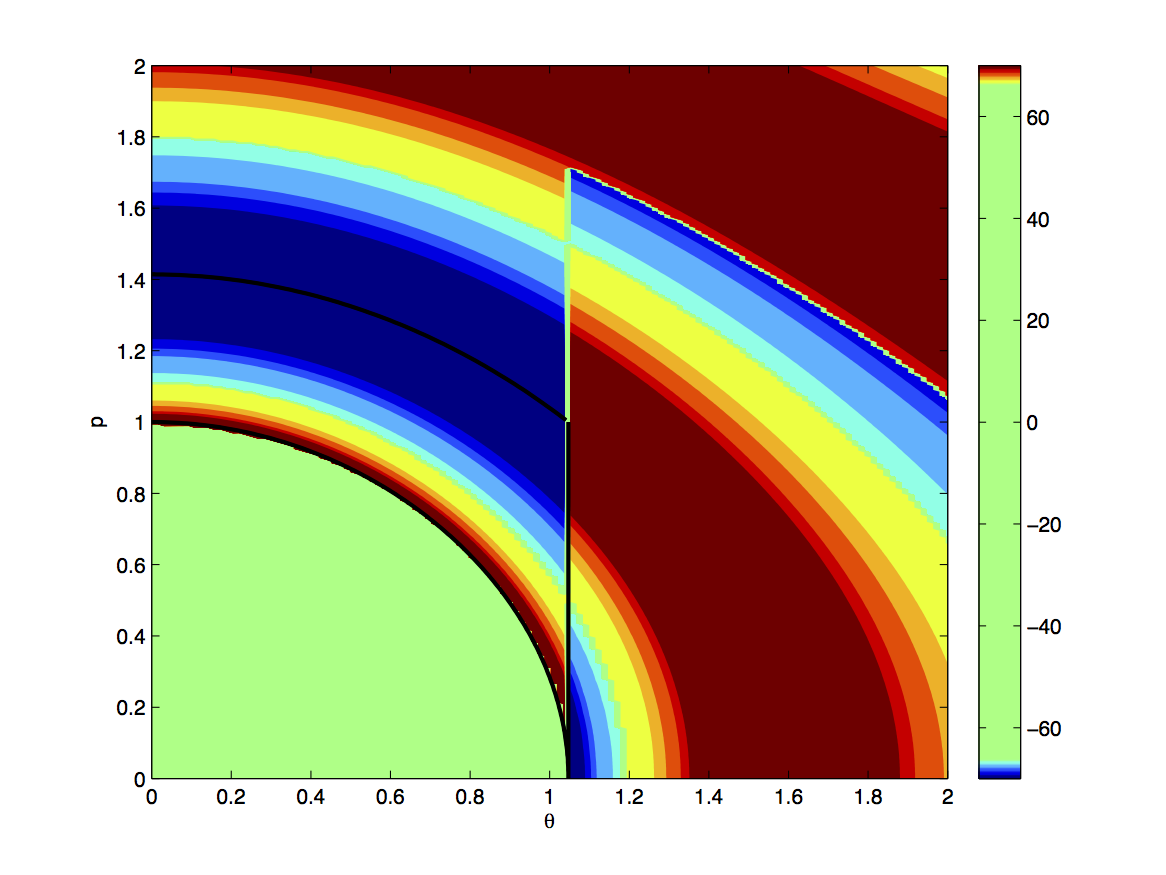}
	
	\caption{Projections of observables onto the eigenspace of Koopman at $\lambda = 0$. The boundaries of the set of limit cycles \eqref{eq:attractingset} is demarcated in black, $\mu_1 = 1 $ rad/s, $\mu_2 = 1$ rad/s, and $\theta^*=\frac{\pi}{3}$ rad. The top two figures are the results obtained with the Hamiltonian function \eqref{eq:hamdef}, whereas the bottom two are those obtained with the signed Hamiltonian function \eqref{eq:signedHam}. The figures on the right show a close up of the results into the region {[0,2] x [0,2]}. }.  \label{mesocronic}
\end{figure*}

As it was observed from \eqref{eq:levelsetprop}, the level sets of Koopman eigenfunctions at $\lambda=0$ partition the state-space into invariant sets. In fact, the characterization of the eigenspace at $\lambda=0$ is directly related to the ergodic partition \cite{Mezic1999, Mezic2004}. 

To describe this partition for the hybrid pendulum, consider at first the dynamics on the invariant set \eqref{eq:attractingset}. From section~\ref{sec:mapping}, we know that this set is foliated by an uncountable family of limit cycles. Given this property, one can verify that
\begin{equation} \phi_{0}(\theta, p) = \delta(h_1(\theta,p) - I_0),\quad I_0\in(0,1) \label{blabla1}
\end{equation}
with $\delta$ denoting the Dirac delta function, form a collection of eigendistributions at $\lambda = 0$, where $h_1$ specifically refers to \eqref{eq:conj1}. In a certain sense, these distributions are the building blocks of all eigenfunctions at $\lambda = 0$. That is, if $c: (0,1) \mapsto \mathbb{C}$ denotes any (Riemann) integrable function, then the convolution:
$$ \phi_{0}(\theta, p) = \int^1_0  c(I) \delta(h_1(\theta,p) - I) dI \quad \equiv \quad c(h_1(\theta,p))$$
is an eigenfunction at $\lambda=0$. 

The definition of $\phi_{0}$ can be extended to outside of $\mathcal{A}_2$ if all initial conditions belonging to the basin of a particular limit cycle $\{p_1,p_1 -1\}$ are assigned the value $c(p_1)$. Clearly, if $c$ is a bijection, the level sets of $\phi_{0}$ separates the basins of every limit cycle.

A method to find eigenfunctions directly from the time-histories of observables involves computing their infinite-time averages \eqref{eq:avg}. These averages are projections of observables \eqref{eq:laplace} onto the eigenspace at $\lambda = 0$, and for observables that are integrable on the set $\mathcal{A}_2$, these averages are well-defined almost everywhere in the kicked region (follows directly from theorem~\ref{thm:main}). 

The top two figures in Fig. \ref{mesocronic} depict a high-resolution contour plot of a projection $\mathcal{P}^0 g$, when $g$ is set equal to the Hamiltonian function (\ref{eq:hamdef}). In the figures, states that have the same color belong to the same level set, and hence, fall under the same equivalence class of long-term dynamical behavior. Note that for this particular eigenfunction, these equivalence classes are not the actual limit cycles themselves, since trajectories that end up in the limit cycles:
$$\{p_1,p_1 -1\} \mbox{ and } \{1-p_1, - p_1\},  \quad p_1\in (0, 1/2)$$ 
have exactly the same time-average.

To obtain more refined partition, one generally needs to consider the product partition of multiple projections conjointly \cite{Mezic}. For the pendulum however, we may also determine the time-average of the observable:
\begin{equation}
g(\theta,p) =  \mbox{sign}(p) \mathcal{H}(\theta,p) \label{eq:signedHam}
\end{equation}
which gives the Hamiltonian a sign, depending on the direction in which the pendulum is moving. This observable is capable of separating the limit cycles $\{p_1,p_1 -1\}$ and $\{1-p_1, - p_1\}$ from each other. The bottom two figures of  Fig~\ref{mesocronic} show contour plots of the projections obtained with this observable.

\subsection{Spectral decomposition on the set of limit cycles} \label{sec:imaginaryspectrum}

We have pointed out that the hybrid pendulum is measure-preserving on the  set \eqref{eq:attractingset}, and an invariant measure is given by \eqref{eq:invmeas}. A consequence of this property is that the Koopman operator is unitary in $L^2(\mathcal{A}_2, \mu_{\mathcal{A}_2})$ and observables belonging to this space of functions admit a spectral decomposition with purely imaginary spectrum \cite{koopman1931hamiltonian,petersen1989ergodic}.

The derivation of the spectral decomposition is most conveniently obtained by first deriving the decomposition for the conjugate system \eqref{eq:flowactionangle} and then applying the conjugacy property of section~\ref{sec:Koopman_intro}. The evolution of a square-integrable function $g(I,\psi)\in L^2(Y, \mu_{Y})$ under the action of the Koopman operator is given by,
$$
\left[ \Koop^t g \right]   (I,\psi) = g(I,\psi+\Omega[I]t \mod 2\pi).
$$
By expanding the observables in a Fourier series  we obtain:
\begin{eqnarray*}
	\left[ \Koop^t g \right](I,\psi) & = & \Koop^t \left( \sum_{j\in\mathbb{Z}}g_j(I) e^{ij\psi} \right) \\
	& = & 
	\sum_{j\in\mathbb{Z}}g_j(I)e^{ij\Omega(I)t} e^{ij\psi} \\
	& = & \int^1_0  g_0(I) \delta(s-I) ds + \\
	& &\sum_{j\in \mathbb{Z},j\neq 0} \int_{\mathbb{R}} e^{i j \rho t} g_j(I) e^{ij\psi}\delta(j\Omega(I)-\rho)d\rho
\end{eqnarray*}
The spectral expansion can be written in the form:
\begin{equation}
\left[ \Koop^t g\right](I,\psi) = g^*(I)+\int_\mathbb{R}e^{i\rho t}dE(\rho) g(I,\psi) \label{eq:expaction}
\end{equation}
where the time average $g^*$ only has dependence on $I$ and where the projection-valued measure $dE(\rho)$ has the explicit expression:
$$
dE(\rho)g(I,\psi)=\sum_{j\in \mathbb{Z},j\neq 0}g_j(I)e^{ij\psi}\delta(j\Omega(I)-\rho)d\rho
$$
Substitution of \eqref{eq:bijection} into \eqref{eq:expaction} will yield the spectral expansion in the original coordinates. Note that the spectrum of the operator turns out to be continuous.

\section{The damped hybrid pendulum: classical geometric analysis}
\label{sec:damping}

In this section, the damped hybrid pendulum is studied from the geometric perspective. 

\subsection{Poincar\'{e} map for the damped system} \label{sec:geomdamp}

The asymptotic properties of the hybrid pendulum under viscous damping can again be analyzed through the study of some discrete map. Because of dissipation,  all trajectories that start at $\mathcal{H}(\theta,p) \geq \mathcal{H}(\theta^*,0)$ must eventually enter the set \eqref{eq:attractingset}. Given that all trajectories in $\mathcal{A}_1$ spiral into the stable fixed point of the pendulum, the analysis of what happens in $\mathcal{A}_2$ is critical to determining the global properties of the system overall. 

The analysis is pursued by viewing the system in terms of its energy state $H$, which has a one-to-one correspondence with the normalized momentum $p$ if we exploit the inherent symmetry in the system. The discrete map that characterizes the dynamics inside  $\mathcal{A}_2$ is defined in reference to the value of the energy at the condition $\left(\pm\theta^*,p\right)$, where $0 \le p \le 1$. In the most general sense, the map takes the form: 
\begin{equation} 
H'  = u\left(H\right) := f \circ d(H) \label{eq:dampmap}
\end{equation} 
Here, $u$ is a composition of two separate functions: $d$, which represents the energy dissipated  due  to damping as the pendulum traverses from $\pm\theta^*$ to $\mp\theta^*$, and $f$, which represents the energy change related to the kicking of the pendulum at $\pm\theta^*$. By knowing the change in momentum that occurs after a kick, and given \eqref{eq:hamdef},  one can derive that:
\begin{equation}
f(H) = H - \left( \frac{\mu_1}{\mu_2}\right)^2 p(H) + \frac{1}{2}\left(\frac{\mu_1}{\mu_2}\right)^2 \label{eq:fdef}
\end{equation}
where:
$$
p(H) = \frac{\mu_2}{\mu_1}\sqrt{2(H + \cos \theta^* -1)}
$$
The dissipation function $d$, on the other hand, must be a monotonically increasing function and takes the form $d\left(H\right) = r\left(H\right)H$ with $0<r\left(H\right)<1$ for all $H>0$.  Based on some numerical simulations (see figure~\ref{fig:dampingfigure}), we make a specific \emph{assumption} that the dissipation function is linear: 
$$d(H) = r H, \quad 0<r<1 $$
\begin{figure}[h!]
	\centering
	\includegraphics[width=0.7\textwidth]{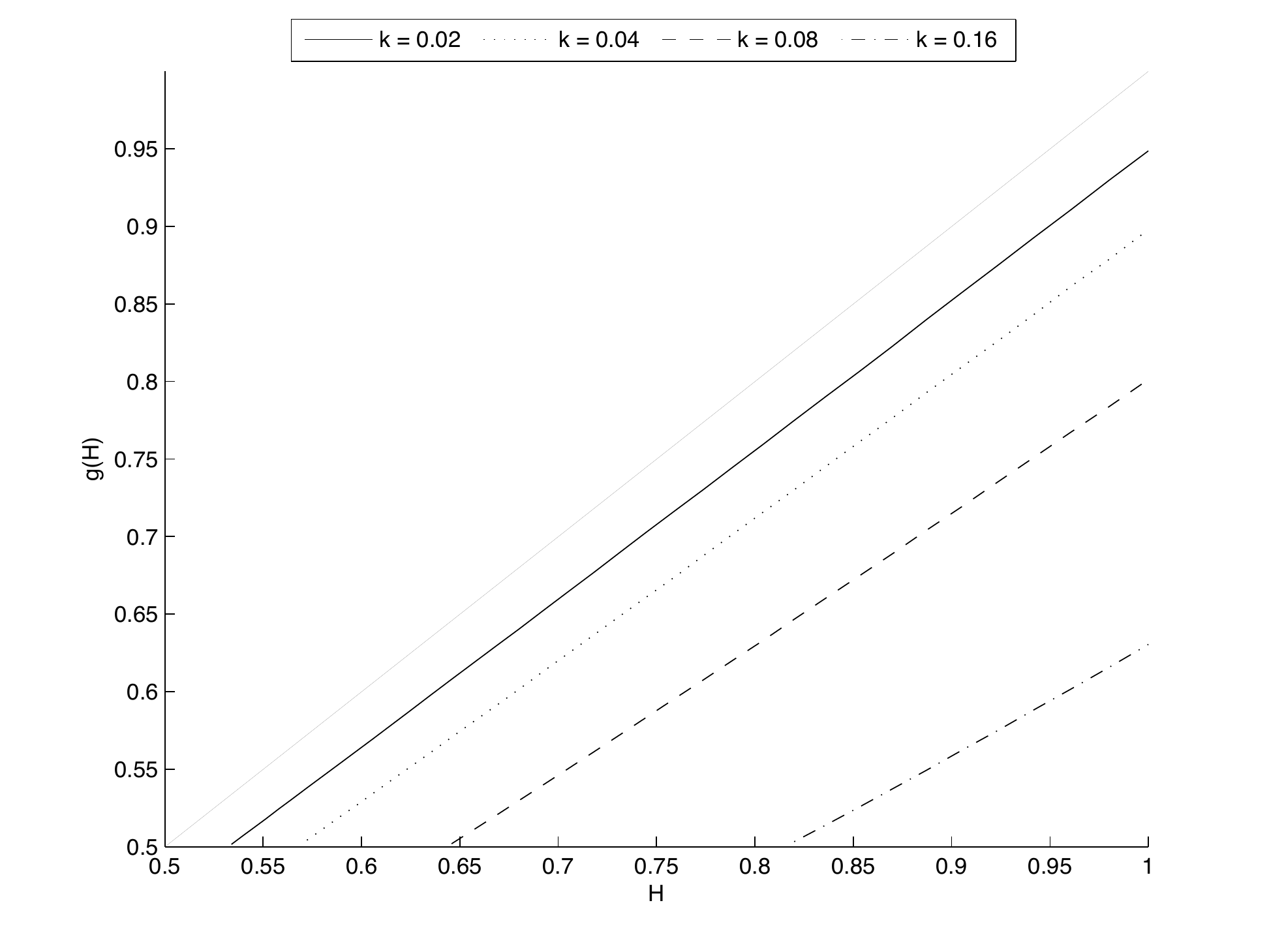}
	\caption{The dissipation function $d$ computed for different viscous damping coefficients. The kick angle $\theta^*$ is set to  $\pi/3$ rad, $\mu_1 = 1$ rad/s, and $\mu_2 = 1$ rad/s.}
	\label{fig:dampingfigure}
\end{figure}

\subsection{Main result} 

In the undamped case, the map \eqref{eq:dampmap} is defined on the domain $[H_-, H_+]$ (see \eqref{eq:defHplusmin}) and has a neutrally stable fixed point at $H_{fp,0} = \mathcal{H}(\theta^*,1/2)$. Furthermore, it has an uncountable family of neutrally stable period-2 cycles , since $f^2 := f \circ f = \mathrm{Id}$. The introduction of damping restricts the domain of the map \eqref{eq:dampmap} to $[H_0, H_+]$, where $H_0 := H_-/r$. The range  of \eqref{eq:dampmap} equals $u([H_0, H_+]) = [H_-, H_+]$, hence a subset of initial conditions will eventually get mapped outside the domain of $u$. The next theorem shows that this subset is, in fact, the entire domain except for one specific point corresponding to the unstable fixed point of $u$.
\begin{theorem} \label{thm:viscousdamp}
	Consider the map \eqref{eq:dampmap} and assume that the dissipation function is linear.
	Furthermore, denote  $r =1 -\delta$, with $\delta>0$ sufficiently small. Then:
	\begin{enumerate}[(i)]
		\item there exists a unique fixed point $H_{fp}(\delta)>H_{fp,0}$ which is  unstable.
		\item the map $u:=f \circ d$, defined by \eqref{eq:dampmap}, has no period 2-cycles.
		\item $\forall H\ne H_{fp}(\delta)$, $\exists n>0$ s.t. 
		$u^n(H) \notin [H_0, H_+]$
	\end{enumerate} 
\end{theorem}
\begin{proof}
	For notational convenience, let us denote $\eta := \frac{\mu_1}{ \mu_2}$. To show that the fixed point $H_{fp}$ moves to the right and observe that it must belong to the graph of the implicit function:
	\begin{displaymath}
	s(H;\delta) := \delta H - \eta^2 p((1-\delta)H) + \frac{1}{2}\eta^2 = 0
	\end{displaymath}
	Since,
	\begin{eqnarray*}
		\left. \frac{\partial }{\partial \delta} s(H;\delta) \right|_{\delta = 0, H = H_{fp,0}} & = & H_{fp,0} \left(1+ \frac{1}{p(H_{fp,0})}   \right) \\
		& = & 3 H_{fp,0} \\
		& > & 0 
	\end{eqnarray*}
	the implicit function theorem guarantees that we may express $H_{fp} = H_{fp}(\delta)$, such that $s(H_{fp}(\delta); \delta) = 0$ for a sufficiently small neighborhood around $\delta = 0$. Furthermore,
	\begin{displaymath}
	H_{fp}(\delta) = H_{fp,0} + 3 H_{fp,0} \delta + \mathcal{O}(|\delta|^2)
	\end{displaymath}
	which shows that the fixed point moves to the right for $\delta>0$. The fixed point $H_{fp}(\delta)$, $\delta>0$ is unstable. To show this, take note that $u(H) = u(H;\delta)$, and:
	\begin{displaymath}
	u'(H;\delta) = (1-\delta) \left( 1 - \frac{1}{p((1-\delta)H)} \right)
	\end{displaymath}
	Consider $u'(H_{fp}(\delta);\delta)$, the fixed point is unstable if, and only if, $| u'(H_{fp}(\delta);\delta)| > 1$. Given that $u'(H_{fp}(0);0) = -1$, we need to show the following:
	\begin{displaymath}
	u'(H_{fp}(\delta);\delta) < u'(H_{fp}(0);0) = -1
	\end{displaymath}
	Since, 
	\begin{eqnarray*}
		\left. \frac{d}{d\delta} u' (H_{fp}(\delta), \delta) \right|_{\delta = 0} & = & \frac{8}{\eta^2} \left( \cos \theta^* - 1\right) \\
		& < & 0,  \quad \mbox{if }\theta^* \ne 0  
	\end{eqnarray*}
	this is indeed the case, which completes the proof for (i).
	
	To prove (ii), consider $u^2(H):= u \circ u(H)$, we will show that for a sufficiently small $\delta>0$,
	$$
	\frac{d}{dH} u^2(H) = u'(u(H)) u'(H) \geq 1 
	$$
	Since $u^2(H_{fp}) = H_{fp}$ with $\frac{d}{dH} u^2(H_{fp})>1$, the above inequality  would imply that $u^2(H)$ has no other fixed points, which proves the non-existence of period-2 cycles. To show \eqref{eq:inequality} holds true, observe at first that $u'(H) = u'(H;\delta)$ and that $u'(u(H)) = u'(u(H;\delta);\delta)$. Furthermore, 
	\begin{eqnarray*}
		u'(H;\delta) & = & u'(H;0) + \left. \frac{\partial}{\partial \delta} u'(H;\delta) 	\right|_{\delta = 0} \delta + \mathcal{O}(|\delta|^2) \\
		u'(u(H;\delta);\delta) & = &  \frac{1}{u'(H;0)} + \left.  \frac{\partial}{\partial \delta}  u'(u(H;\delta);\delta)	\right|_{\delta = 0} \delta  + \mathcal{O}(|\delta|^2)
	\end{eqnarray*}
	So that for $\delta>0$. sufficiently small, we have:
	$$
	\frac{d}{dH} u^2(H) = 1 +   \left( \left. \frac{\partial}{\partial \delta} u'(H;\delta) 	\right|_{\delta = 0} \delta  \right) \left(  \frac{1}{u'(H;0)} \right)  + \left(  \left.  \frac{\partial}{\partial \delta}  u'(u(H;\delta);\delta)	\right|_{\delta = 0} \delta \right)  u'(H;0) + \mathcal{O}(|\delta|^2)
	$$
	Hence, if it can be shown that:
	$$
	k(H) :=  \left. \frac{\partial}{\partial \delta} u'(H;\delta) 	\right|_{\delta = 0} \cdot  \frac{1}{u'(H;0)}   + \left.  \frac{\partial}{\partial \delta}  u'(u(H;\delta);\delta)	\right|_{\delta = 0}  u'(H;0) \geq 0
	$$
	then $\frac{d}{dH}u^2(H)\geq1$. And indeed, 
	\begin{eqnarray*}
		\left. \frac{\partial}{\partial \delta} u'(H;\delta) 	\right|_{\delta = 0}  & = & - u'(H;0) - \frac{\eta H}{(\eta p(H))^3} \\
		\left.  \frac{\partial}{\partial \delta}  u'(u(H;\delta);\delta)	\right|_{\delta = 0} & = & -\frac{1}{u'(H;0)}  - \frac{\eta u(H;0) }{(\eta (1- p(H)))^3} 
	\end{eqnarray*}
	and after some algebraic manipulations, we obtain:
	\begin{eqnarray*}
		k(H) & = &-2 + \frac{\frac{1}{2}\eta^2  p(H) (1-p(H)) + 1 - \cos \theta^* }{\eta^2 p(H)^2 (1-p(H))^2} \\
		& \geq & -2  + \frac{1}{2} \frac{  1}{ p(H) (1-p(H))} \\
		& \geq & -2 + \frac{1}{2}\displaystyle \min_{0<p<1} \frac{  1}{ p (1-p)} \\
		& = & 0
	\end{eqnarray*}
	which completes the proof for (ii).
	
	The proof of (iii) follows directly from $\frac{d}{dH}u^2(H)\geq1$.
\end{proof}
\begin{corollary} \label{cor:damped}
	For $\delta>0$ sufficiently small, the trajectories of the hybrid pendulum, starting from almost everywhere, asymptotically reach the fixed at $\theta = 0$.
\end{corollary}

\section{The damped hybrid pendulum: Koopman analysis} \label{sec:dampedkoopman}

In this section, the damped hybrid pendulum is analyzed from the Koopman operator theory perspective.

\subsection{Eigenspace of Koopman at $\lambda=0$} \label{sec:eigzerodamp}

From corollary~\ref{cor:damped}, it follows that almost all trajectories end up at the stable equilibrium of the pendulum as time approaches infinity. This has specific implications to the eigenspace at $\lambda=0$ if one restrict the Koopman operator to those functions\footnote{Note that for the damped hybrid pendulum, this space of functions is an invariant subspace of the operator.} which are continuous at $(\theta,p) =(0,0)$
The introduction of damping, in that case, would severely simplify the eigenspace at $\lambda=0$, given that the only permissible eigenfunctions are now those which are constant almost everywhere in $X$.

\subsection{Point spectrum of the operator} \label{sec:disseig}

\begin{figure*}[t!]
	\centering
	\subfigure[The modulus $|\phi_{\lambda}(\boldsymbol{x})|$]{\includegraphics[width=.45 \textwidth]{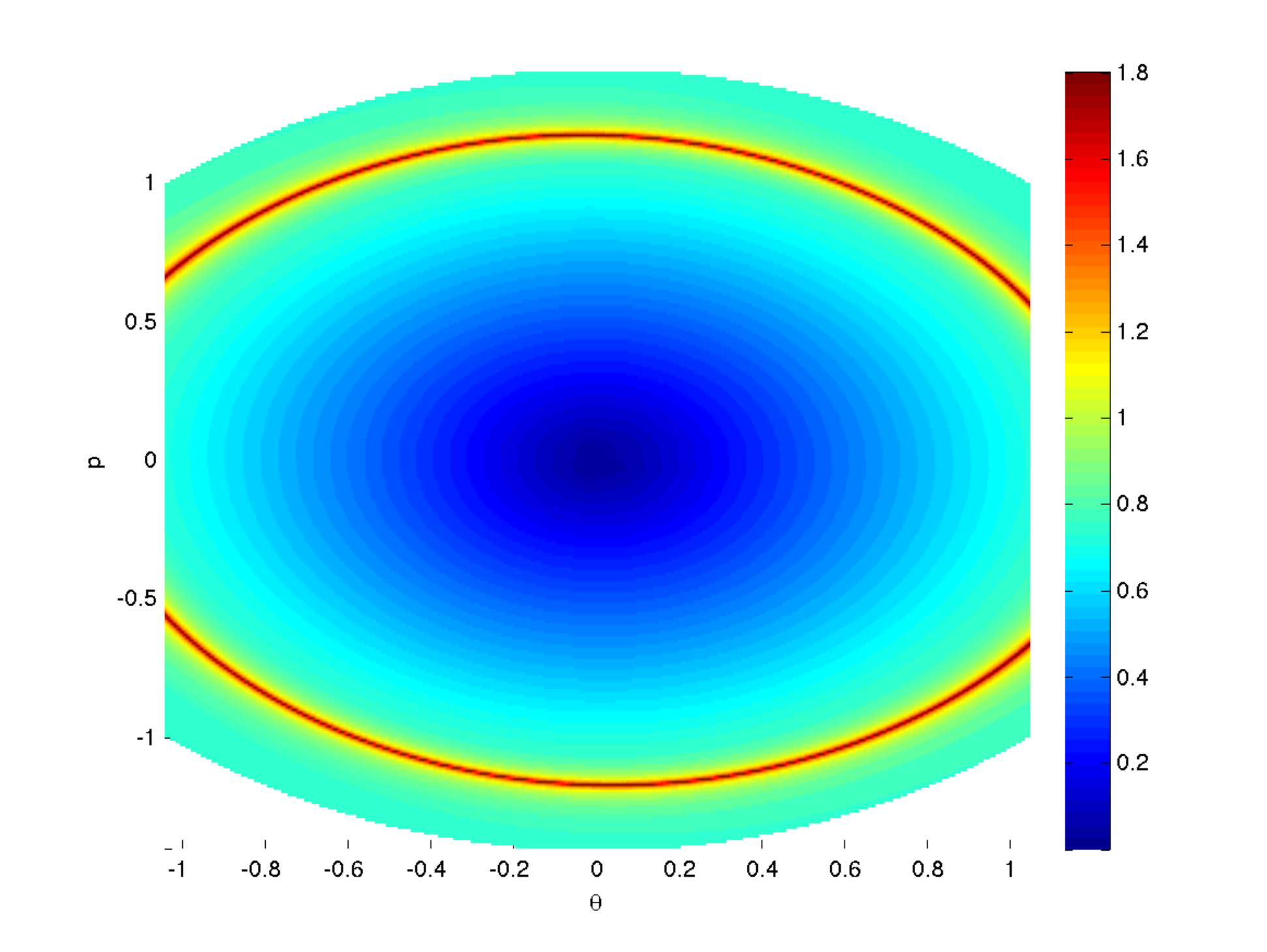} \includegraphics[width=.55 \textwidth]{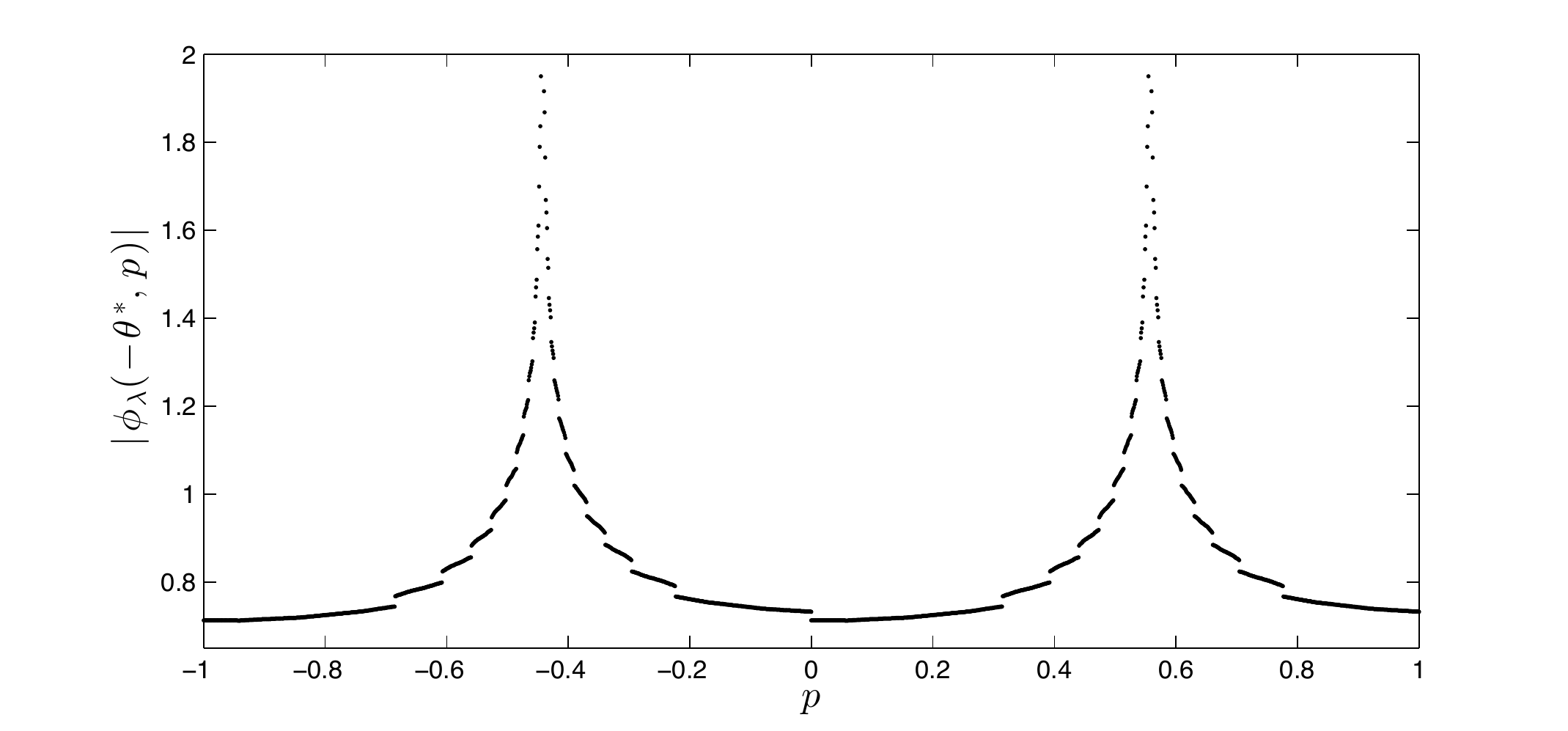} }  \label{modphi} \\
	\subfigure[The phase $\angle \phi_{\lambda}(\boldsymbol{x})$]{\includegraphics[width=.45 \textwidth]{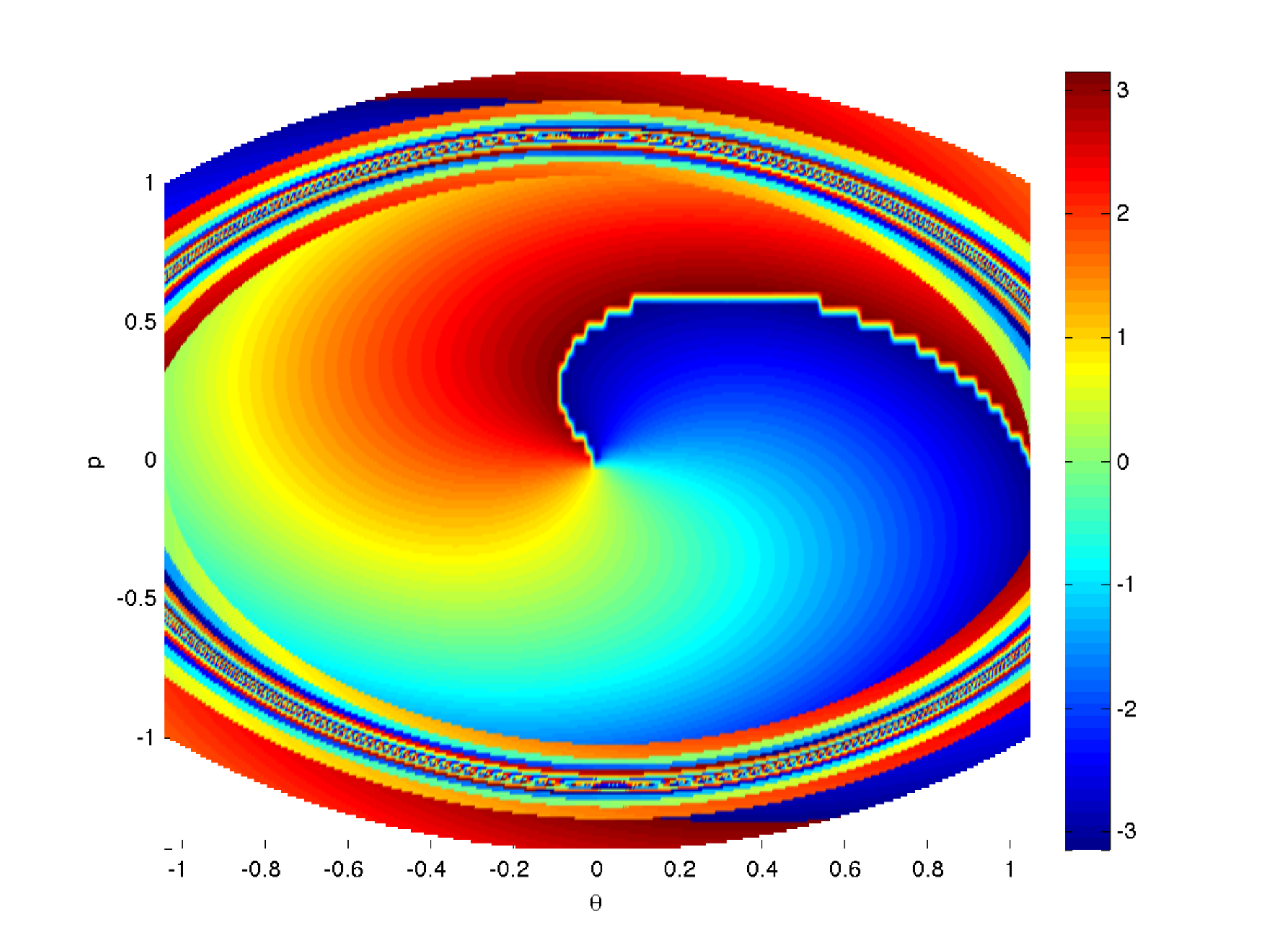} \includegraphics[width=.55 \textwidth]{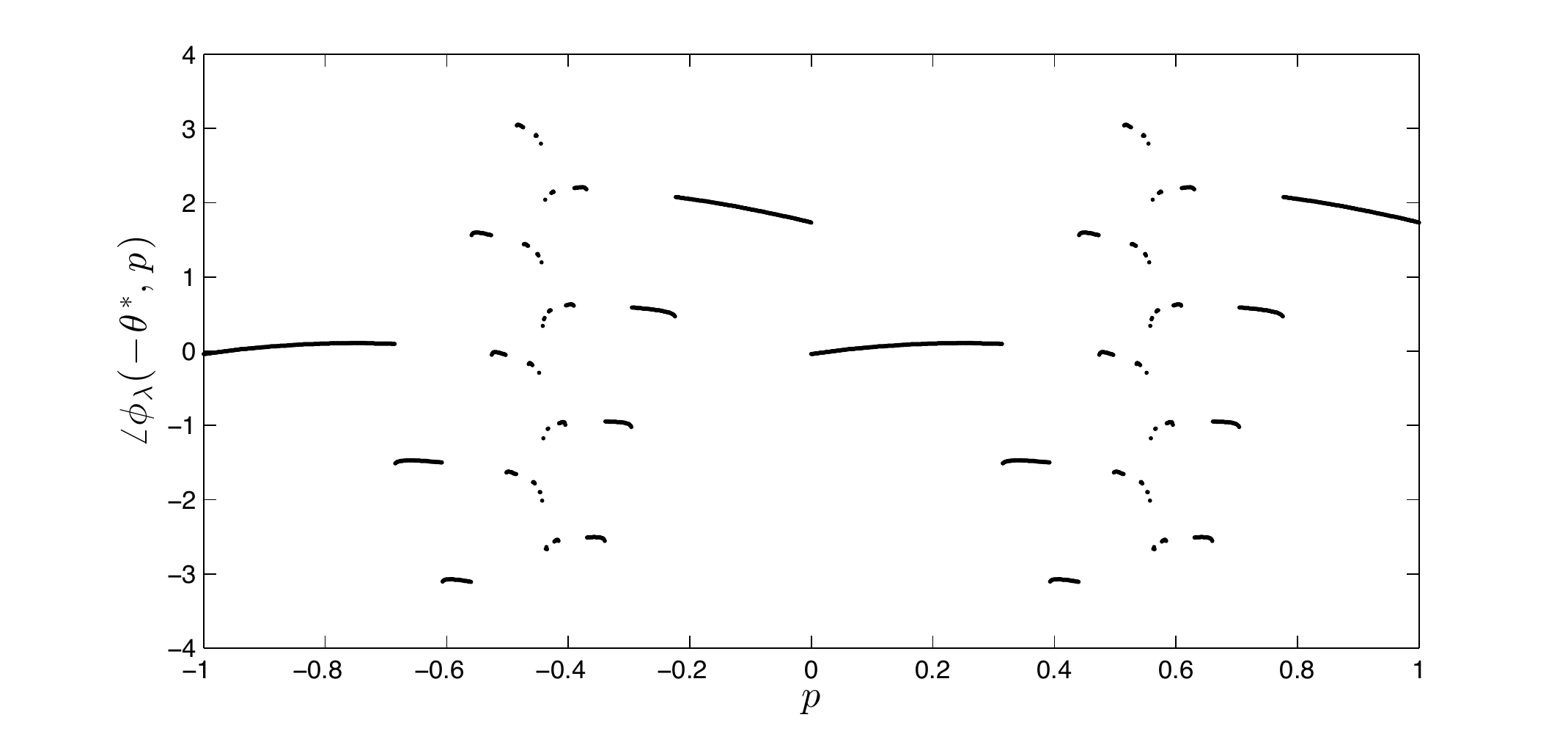}  } 
	\caption{The Koopman eigenfunction $\phi_{\lambda}(\boldsymbol{x})$ of theorem~\ref{thm:dampeig} shown for the regions \eqref{eq:simple} and \eqref{eq:attractingset}. The figures on the right display the eigenfunction at a cut: $\theta = -\theta^*$, $p\in(-1,1)$. The viscous damping coefficient is set to $k= 0.03$.   }  \label{fig:dampedeig}
\end{figure*}

The addition of viscous damping turns the fixed point at $\theta = 0$ into a spiral sink. In terms of the Koopman operator, these changes give rise to point spectrum in the left-half complex plane. The point spectra are products of the eigenvalues\footnote{This follows from the property $ \Koop g_1 g_2 = \left( \Koop g_1 \right)  \left( \Koop g_2 \right)$, see  \cite{Budisic2012} for details. } of the linearized pendulum:
\begin{equation}
\boldsymbol{\dot{y}} = \mathrm{A} \boldsymbol{y}, \quad \mathrm{A} = \begin{bmatrix} 0 & 1 \\ -1 & -k \end{bmatrix} \label{eq:linsys}
\end{equation} 
where $ \boldsymbol{y} = \begin{bmatrix} \theta & p \end{bmatrix}^T$ and whose eigenvalues are given by $\lambda / \bar{\lambda} = -\sigma \pm i \eta$, with $\sigma = \frac{1}{2}k$, $\eta = \sqrt{1- \frac{1}{4} k^2}$. 

Following the concepts discussed in \cite{Lan2013,Mauroy2013}, one can show that an eigenfunction at $\lambda $ and $\bar{\lambda}$ can be computed from the observables:
\begin{eqnarray}  g_1(\theta,p) & := &  \frac{1}{\sqrt{2}} \left \Vert \begin{bmatrix}\boldsymbol{v}& \boldsymbol{\bar{v}} \end{bmatrix}^{-1}\begin{bmatrix} \theta \\ p \end{bmatrix} \right \Vert_{2}, \\
g_2(\boldsymbol{y}) & := &  \frac{\begin{bmatrix}1 & 0 \end{bmatrix}\begin{bmatrix}\boldsymbol{v}& \boldsymbol{\bar{v}} \end{bmatrix}^{-1} \begin{bmatrix} \theta \\ p \end{bmatrix} }{g_1(\theta, p)}
\end{eqnarray}
where $\boldsymbol{v}$, $\boldsymbol{\bar{v}}$ are the right eigenvectors of $\mathrm{A}$. 

\begin{theorem} \label{thm:dampeig} Consider the damped hybrid pendulum defined by \eqref{eq:quardnorm}, \eqref{eq:penddamped}. Then,
	\begin{equation} \phi_{\lambda / \bar{\lambda} }(\theta, p) = | \phi_{\lambda}(p,\theta) | e^{\pm i \angle \phi_{\lambda}(\theta,p) } 
	\end{equation}	
	with:
	\begin{subequations}
		\begin{eqnarray}
		| \phi_{\lambda}(\theta,p) | & := & \displaystyle \lim_{t\rightarrow \infty} \frac{1}{t} \int^{t}_{0}  e^{\sigma \tau} \left[\Koop^{\tau}g_1\right](\theta,p) d\tau \\
		e^{\pm i \angle \phi_{\lambda}(\theta, p) }  & := &  \displaystyle \lim_{t\rightarrow \infty} \frac{1}{t} \int^{t}_{0}  e^{\pm i \eta \tau} \left[\Koop^{\tau}g_2\right](\theta,p) d\tau\mbox{  } 
		\end{eqnarray} \label{eq:dampeigint}
	\end{subequations}
	are Koopman eigenfunctions at eigenvalues $\lambda / \bar{\lambda} = -\sigma \pm i \eta$.
\end{theorem}
\begin{proof}
	The claim follows by showing that the integrals (33) converge (to a non-zero value) for all initial conditions on the basin. If $(\theta, p) \in \mathcal{A}_1$, the dynamics are identical to that of the conventional pendulum, and therefore by Theorem 2.3 in \cite{Lan2013},  there exists a $C^1$-diffeomorphism $h:\mathcal{A}_1 \mapsto Y \subset \mathbb{R}^2$ between the flows $\mathbf{S}^t$ and $\mathbf{R}^t$. For all other initial conditions on the basin, we infer from theorem 2 and corollary 1  that there exists a $T^*>0$ such that:
	$$S^t(\theta, p) \in \mathcal{A}_1, \quad \forall t> T^*$$
	A change of variables may be used to prove convergence of the integrals in that case.
\end{proof}

Figure~\ref{fig:dampedeig} shows a contour plot of the eigenfunction in theorem~\ref{thm:dampeig}. The functions $| \phi_{\lambda}(\theta,p) |$ and $e^{i \angle \phi_{\lambda}(\theta, p) }$ have the following geometric interpretation. The level sets of $| \phi_{\lambda}(\theta,p) |$ define the so-called isostables \cite{Mauroy2013} and describe the set of points that have the same asymptotic convergence toward the fixed point. We see particularly that the isostables blow up in the region that corresponds to the unstable periodic orbit (associated with the fixed point in theorem~\ref{thm:viscousdamp}).  The level sets of $e^{i \angle \phi_{\lambda}(\theta,p) }$ (or equivalently those of $\angle \phi_{\lambda}((\theta,p)$), on the other hand,  describe the set of points that simultaneously move in phase around the fixed point.

A close examination of the contour plots suggests that merging of trajectories indeed occur for certain initial conditions in \eqref{eq:attractingset}. Additionally, the phase plots indicate that the kicking of the pendulum introduces a high level of phase sensitivity \cite{Mauroy2015} close to the unstable periodic orbit.   

Overall, the eigenfunctions of theorem~\ref{thm:dampeig} can be used to describe a \emph{semi-conjugacy} with a linear system. Specifically, the modulus and phase form a map $(\theta,p) \mapsto (| \phi_{\lambda}(\theta,p) |, \angle \phi_{\lambda}(\theta,p) )$, such that under the new coordinates we have the simplified dynamics:
\begin{eqnarray*}
	\frac{d}{dt} | \phi_{\lambda}(\theta,p) | & = & -\sigma  | \phi_{\lambda}(\theta,p) |  \\
	\frac{d}{dt} \angle \phi_{\lambda}(\theta,p)  & = & \eta 
\end{eqnarray*}

\section{Conclusions}
\label{sec:conclusions}
The spectral properties of the Koopman operator are closely related to the geometric properties of the state-space, and in this paper, we have discussed in detail how these relationships exactly manifest for the hybrid pendulum. The connections between level sets of Koopman eigenfunctions and the corresponding flow field were useful for visualizing certain geometric properties of the state-space. In the undamped case, the ergodic partition, obtained by projecting of observables onto the eigenspace at zero, yielded a method to visualize the basins of attraction of the limit cycles. In the damped case, the eigenfunctions associated with the spectra in the left-half complex-plane, provided a set of coordinates to establish a semi-conjugacy with a linear system.

\bibliographystyle{plain}
\bibliography{koopmanbib2015}

\end{document}